\documentclass[11pt]{amsart}
\usepackage{latexsym}
\usepackage{amssymb}
\usepackage{amsfonts}
\usepackage{graphicx}
\usepackage{epsf}

\newtheorem{theorem}{Theorem}[section]
\newtheorem{corollary}[theorem]{Corollary}
\newtheorem{lemma}[theorem]{Lemma}
\newtheorem{proposition}[theorem]{Proposition}
\newtheorem{definition}[theorem]{Definition}
\newtheorem{example}[theorem]{Example}
\newtheorem{remark}[theorem]{Remark}

\def\bea{\begin{eqnarray*}}
\def\eea{\end{eqnarray*}}

\def\ot{\otimes}
\def\ra{\rightarrow}

\def\al{\alpha}

\def\bea{\begin{eqnarray*}}
\def\eea{\end{eqnarray*}}

\begin{document}
\title{Graded semisimple algebras are symmetric}
\author{
S. D\u{a}sc\u{a}lescu$^{1,2}$, C. N\u{a}st\u{a}sescu$^{1,3}$ and
L. N\u{a}st\u{a}sescu$^{1,3}$}
\address{$^1$ University of Bucharest, Faculty of Mathematics and Computer Science,
Str. Academiei 14, Bucharest 1, RO-010014, Romania}
\address{$^2$ ICUB, University of Bucharest}\address{ $^3$ Institute of Mathematics of the Romanian
Academy, PO-Box
1-764\\
RO-014700, Bucharest, Romania}
\address{
 e-mail: sdascal@fmi.unibuc.ro, Constantin\_nastasescu@yahoo.com,
 lauranastasescu@gmail.com
}

\date{}
\maketitle

\begin{abstract}
We study graded symmetric algebras, which are the symmetric
monoids in the monoidal category of vector spaces graded by a
group. We show that a finite dimensional graded semisimple algebra
 is graded symmetric. The center of a symmetric algebra is not necessarily symmetric, but we prove that the center of a finite
dimensional
graded division algebra is symmetric, provided that the order of the grading group is not divisible by the characteristic of the base field. \\
2010 MSC: 16W50, 16K20, 16S35, 18D10\\
Key words:  graded algebra, Frobenius algebra, symmetric algebra,
graded division algebra, crossed product, graded semisimple
algebra.
\end{abstract}

\section{Introduction and preliminaries}
Frobenius algebras are algebraic objects that arose from
representation theory of groups, but they are also present in
other areas of mathematics, for example in the theory of Hopf
algebras and quantum groups, in the theory of compact oriented
manifolds, in topological quantum field theory, etc., see
\cite{kadison}. A finite dimensional algebra $A$ over a field $k$
is Frobenius if $A$ and its $k$-dual $A^*$ are isomorphic as left
$A$-modules, or equivalently, as right $A$-modules. There are
certain Frobenius algebras having more symmetry and also a rich
representation theory; these are the symmetric algebras. $A$ is
called symmetric if $A$ and $A^*$ are isomorphic as
$A,A$-bimodules, or equivalently, if there exists a linear map
$\lambda:A\ra k$ such that $\lambda(ab)=\lambda(ba)$ for any
$a,b\in A$, and ${\rm Ker}\; \lambda$ does not contain non-zero
left ideals. It was proved by Eilenberg and Nakayama in their
pioneering work \cite{en} on Frobenius algebras that a finite
dimensional semisimple algebra is symmetric. An equivalent
characterization of Frobenius algebras, given in \cite{abrams2},
opened the way to study Frobenius algebras in monoidal categories,
which was initiated in \cite{muger}, \cite{street}, see \cite{bt}
for more recent developments. Symmetric algebras have also been
considered in monoidal categories with more structure (existence
of duals), for example in sovereign categories, see \cite{fuchs}.
It seems to be an interesting problem to understand the structure
of Frobenius (symmetric) algebras in certain monoidal categories.
A study in this direction was done in \cite{dnn} for the category
of vector spaces graded by an arbitrary group $G$.  A finite
dimensional $G$-graded algebra $A=\oplus A_g$ is symmetric in this
category (we shortly say that $A$ is graded symmetric), if $A$ and
$A^*$ are isomorphic as $G$-graded $A,A$-bimodules. Here the
grading on $A^*$ is given by $(A^*)_g=\{ f\in A^*\; |\; f(A_h)=0
\; \mbox{for any } h\neq g^{-1}\}$. Graded symmetric algebras were
considered in \cite{harris} in the case of strongly graded
algebras, and in \cite[Section 3]{marcus} in the case of graded
crossed products. Taking into account what happens in the monoidal
category of vector spaces, it is a natural question to ask whether
finite dimensional graded semisimple algebras are graded
symmetric, i.e. symmetric monoids in the category of graded vector
spaces. Our main result is\\

{\bf Theorem A.} {\it A finite dimensional graded semisimple
algebra is graded symmetric.}\\

The first step is to look at graded division algebras. After we
make some general considerations about graded symmetric graded
crossed products in Section \ref{sectioncrossedproducts}, we apply
them to graded division algebras in Section
\ref{sectiondivisionalgebras}. If $A$ is a finite dimensional
$G$-graded division algebra, and $D$ is the homogeneous component
of degree $e$ of $A$, where $e$ is the neutral element of $G$, we
construct an action of $G$ on the space $D/[D,D]$, and prove that
$A$ is graded symmetric if and only if $D/[D,D]$ has non-zero
coinvariants. By using Hilbert Theorem 90 we show that $D/[D,D]$
is isomorphic as a $kG$-module to the center $\ell$ of $D$, which
is itself a $kG$-module. Finally, $\ell$ has non-zero coinvariants
by using the Normal Basis Theorem. Next we transfer the result to
graded simple algebras using the graded version of Wedderburn's
Theorem, and further to graded semisimple algebras.

Theorem A is of particular interest in the case of graded division
algebras, which occur in graded Clifford theory. Indeed, if $R$ is
a group graded algebra, and $\Sigma$ is a simple object in the
category of graded left $R$-modules, then the full additive
subcategory generated by $\Sigma$ in the category of left
$R$-modules (the objects are epimorphic images of direct sums of
copies of $\Sigma$) is equivalent to the category of left
$\Delta$-modules, where $\Delta$ is the endomorphism ring of the
$R$-module $\Sigma$, see \cite{dade} or \cite[Chapter 4]{nvo}.
Since $\Sigma$ is graded simple, $\Delta$ is a graded division
algebra, so then the symmetry of $\Delta$ provides information on
its representation theory, and further on the structure of
$\Sigma$ as an $R$-module.

A result related to Theorem A was proved in \cite[Theorem
4]{harris}: if $G$ is a finite group and $A$ is a strongly
$G$-graded algebra over a commutative ring $k$, such that the
homogeneous component of $A$ of trivial degree is a finite direct
sum of full matrix algebras over $k$, then $A$ is graded
symmetric. There is just little overlap with our result. In the
case where $k$ is a field, the result of \cite{harris} covers only
a limited class of graded semisimple algebras.

We show that the center of a graded division algebra need not be
symmetric. However,  we have the
following. \\

{\bf Theorem B.} {\it Let $A$ be a finite dimensional $G$-graded
division algebra such that ${\rm char}\, k\nmid |G|$. Then the
center of $A$ is a symmetric algebra.}\\

We mention that a question which is in some sense dual, was
discussed in \cite{nakayama}: when is a quotient of a symmetric
algebra symmetric? We refer to \cite{kkl} for some recent
developments, related to block theory of group algebras.

We work over a field $k$. If $G$ is a group, a $G$-graded algebra
is a $k$-algebra $A$ with a decomposition of $k$-spaces
$A=\oplus_{g\in G}A_g$ such that $A_gA_h\subseteq A_{gh}$ for any
$g,h\in G$. $A$ is called a graded division algebra if any
non-zero homogeneous element is invertible. In this case $A_e$ is
a division algebra, where $e$ is the neutral element of $G$. A
$G$-graded algebra is called a graded crossed product if any
homogeneous component of $A$ contains an invertible element. The
support $H=\{ h\in G|A_g\neq 0\}$ of a $G$-graded division algebra
$A$ is a subgroup of $G$, and $A$ is obviously a graded crossed
product when regarded as an $H$-graded algebra. A graded algebra
$A$ is called graded semisimple if the category of graded left
$A$-modules is semisimple, i.e. any object is a sum of simple
objects in the category.

For notation and terminology about graded algebras we refer to
\cite{nvo}. We recall from \cite{dnn} that a finite dimensional
$G$-graded algebra $A$ is graded symmetric if and only if there
exists a linear map $\lambda:A\ra k$ such that
$\lambda(ab)=\lambda(ba)$ for any $a,b\in A$, $\lambda(A_g)=0$ for
any $g\neq e$, and ${\rm Ker}\; \lambda$ does not contain non-zero
graded left ideals.

\section{Graded crossed products} \label{sectioncrossedproducts}

Let $A=\oplus_{g\in G}A_g$ be a $G$-graded crossed product. Denote
$A_e=D$. For any $g\in G$ we choose an invertible element $u_g\in
A_g$. We define a weak action of $G$ on $D$, i.e. a map
$\sigma:G\ra Aut_{k-alg}(D)$, by $\sigma(g)(x)=u_gxu_g^{-1}$ for
any $g\in G$ and $x\in D$. Denote $\sigma(g)(x)=g\cdot x$. Of
course, this weak action depends on the choice of the family
$(u_g)_{g\in G}$.

Let $\ell =Cen(D)$, the center of the algebra $D$. Then $\ell$ is
invariant under the action of each $\sigma(g)$, so $\sigma$
induces a weak action of $G$ on $\ell$. This is in fact a usual
action of $G$ on $\ell$, i.e. the map $\theta:G\ra
Aut_{k-alg}(\ell)=Gal(\ell/k)$ induced by $\sigma$ is a group
morphism, since for any $g,h\in G$ and any $a\in \ell$ \bea g\cdot
(h\cdot a)&=&u_gu_hau_h^{-1}u_g^{-1}\\&=&
u_{gh}u_{gh}^{-1}u_gu_hau_h^{-1}u_g^{-1}\\
&=&u_{gh}au_{gh}^{-1}u_gu_hu_h^{-1}u_g^{-1}\\
&=&u_{gh}au_{gh}^{-1}\\&=&(gh)\cdot a.\eea Moreover, this action
of $G$ on $\ell$ does not depend on the choice of the set
$(u_g)_{g\in G}$. Indeed, if $(v_g)_{g\in G}$ is another set of
invertible elements with $v_g\in A_g$ for any $g\in G$, one has
\bea v_gav_g^{-1}&=&u_gu_g^{-1}v_gav_g^{-1}\\
&=&u_gau_g^{-1}v_gv_g^{-1}\\
&=&u_gau_g^{-1}\eea for any $g\in G$ and $a\in \ell$.\\

On the other hand, $[D,D]$ is invariant under the weak action of
$G$, since $\sigma(g)$ is an algebra morphism. Then $\sigma$
induces a weak action of $G$ on the factor space $D/[D,D]$,
defined by $g\cdot \overline{x}=\overline{g\cdot x}$, where
$\overline{x}$ denotes the class modulo $[D,D]$. Since \bea g\cdot
(h\cdot x)&=&u_gu_hxu_h^{-1}u_g^{-1}\\
&=&u_gu_hu_{gh}^{-1}u_{gh}xu_h^{-1}u_g^{-1}\\
&=&(u_gu_hu_{gh}^{-1})(u_{gh}xu_h^{-1}u_g^{-1})-(u_{gh}xu_{h}^{-1}u_g^{-1})(u_gu_hu_{gh}^{-1})+u_{gh}xu_{gh}^{-1}\\
&=&[u_gu_hu_{gh}^{-1},u_{gh}xu_h^{-1}u_g^{-1}]+ (gh)\cdot x\eea
for any $x\in D$ and any $g,h\in G$, we see that $g\cdot (h\cdot
x)-(gh)\cdot x\in [D,D]$, so the weak action of $G$ on $D/[D,D]$
is also an usual action.

\begin{definition} {\rm A symmetric $k$-linear form $\lambda:D\ra k$ is
called $G$-invariant if $\lambda(x)=\lambda(g\cdot x)$ for any
$g\in G$ and any $x\in D$.}
\end{definition}

We note that the definition of $G$-invariance on a symmetric
linear form $\lambda$ does not depend on the choice of
$(u_g)_{g\in G}$. Indeed, if $\lambda$ is $G$-invariant and
$(v_g)_{g\in G}$ is another set of invertible elements, with $v_g$
of degree $g$, then \bea \lambda (v_gxv_g^{-1})&=&\lambda
(g\cdot(v_gxv_g^{-1}))\\&=&\lambda
(u_{g^{-1}}v_gxv_g^{-1}u_{g^{-1}}^{-1})\\&=&\lambda
(xv_g^{-1}u_{g^{-1}}^{-1}u_{g^{-1}}v_g)\\&=&\lambda (x).\eea

\begin{proposition} \label{propsymmetriccrossedproducts}
Let $A$ be a $G$-graded crossed product, and let $A_e=D$. Then $A$
is graded symmetric if and only if there exists a $G$-invariant
symmetric linear form $\lambda:D\ra k$, whose kernel does not
contain non-zero left ideals of $D$.
\end{proposition}
\begin{proof}
Assume that $A$ is graded symmetric, thus there is a non-zero
symmetric linear form $\Lambda:A\ra k$ with $\Lambda(A_g)=0$ for
any $g\neq e$, and such that the kernel of $\Lambda$ does not
contain non-zero graded left ideals of $A$. Let $\lambda:D\ra k$
be the restriction of $\Lambda$ to $D$. This is clearly symmetric.
If $I$ is a left ideal of $D$ with $\lambda(I)=0$, then
$\Lambda(AI)=0$, so $AI$ is zero, and then so is $I$. Moreover,
$\lambda (g\cdot x)=\Lambda (u_gxu_g^{-1})=\Lambda
(xu_{g^{-1}}u_g)=\Lambda(x)=\lambda(x)$ for any $g\in G$ and $x\in
D$, so $\lambda$ is $G$-invariant.\\

Conversely, let $\lambda:D\ra k$ be a $G$-invariant non-zero
symmetric linear form. Define $\Lambda:A\ra k$ by
$\Lambda(a)=\lambda (a_e)$ for any $a\in A$, where $a_e$ is the
homogeneous component of degree $e$ of $a$. If $J$ is a graded
left ideal of $A$ with $\Lambda(J)=0$, then $\lambda(J_e)=0$, so
$J_e=0$. Then $J=AJ_e=0$. Also, $\Lambda$ is symmetric. Indeed,
since $\Lambda (A_g)=0$ for $g\neq e$, it is enough to show that
$\Lambda(ab)=\Lambda(ba)$ for any $a\in A_g$ and $b\in
A_{g^{-1}}$, where $g\in G$. For such $a$ and $b$ we have \bea
\Lambda(ab)&=&\lambda(ab)\\
&=&\lambda(u_{g^{-1}}abu_{g^{-1}}^{-1})\\
&=&\lambda(bu_{g^{-1}}^{-1}u_{g^{-1}}a)\\
&=&\lambda (ba)\\
&=&\Lambda (ba)\eea
\end{proof}

\begin{remark}
{\rm We note that the condition on the $G$-invariance of the
symmetrizing form $\lambda$ on $D$ cannot be omitted in
Proposition \ref{propsymmetriccrossedproducts}. In other words, if
$A_e=D$ is symmetric, it does not necessarily follow that $A$ is
graded symmetric. An example of such a graded crossed product is
given in \cite[Remark 5.3]{dnn}, where a symmetric algebra $R$
with an automorphism $g$ of order 2 is considered, such that the
subalgebra of invariants $R^g$ is not symmetric. Then the skew
group algebra $A=R*C_2$ is a $C_2$-graded crossed product, $A_e=R$
is symmetric, but $A$ is not even symmetric.

This example shows that if $A$ is a finite dimensional strongly
graded algebra such that $A_e$ is symmetric, it does not follow in
general that $A$ is graded symmetric. The converse always holds,
i.e. if $A$ is graded symmetric, then it is clear that $A_e$ must
be symmetric.}
\end{remark}

\section{Graded division algebras}
\label{sectiondivisionalgebras}

Throughout this section $A$ is a finite dimensional $G$-graded
division algebra. Since the support of $A$ is a subgroup of $G$,
there is no loss of generality if we assume that the support of
$A$ is $G$. In particular $A$ is a $G$-graded crossed product, and
we can use the results and notation of Section
\ref{sectioncrossedproducts}. In this case, $D=A_e$ is a division
algebra.

We recall that for a vector space $V$ on which the group $G$ acts
(i.e. $V$ is a left $kG$-module), the dual space $V^*$ has an
induced right $G$-action. A linear form $v^*\in V^*$ is
$G$-invariant if and only if $v^*(g\cdot v-v)=0$ for any $v\in V$
and $g\in G$. Thus the subspace of $G$-invariants $(V^*)^G$ can be
identified with $(V_G)^*$, the dual of the space $V_G$ of
coinvariants of $V$; $V_G$ is the factor space of $V$ with respect
to the subspace spanned by all vectors of the form $g\cdot v-v$,
with $g\in G$ and $v\in V$. In other words, $V_G=V/\omega(G)V$,
where $\omega(G)$ is the augmentation ideal of the group algebra
$kG$.

We use these considerations for $V=D/[D,D]$, with the $G$-action
defined in Section \ref{sectioncrossedproducts}. We note that the
right action induced on $V^*$ is a special case of the
construction performed in \cite[Lemma 1]{harris} for a strongly
graded algebra over a commutative ring.

 Now
we have the following characterization.

\begin{proposition} \label{symmetriccoinvariants}
Let $A$ be a finite dimensional $G$-graded division algebra. Then
the following assertions are equivalent.\\
(1) $A$ is graded symmetric.\\
(2) $Hom(D/[D,D],k)^G\neq 0$.\\
(3) $(D/[D,D])_G\neq 0$.
\end{proposition}
\begin{proof} It is clear that the $k$-linear space of all
symmetric linear forms on $D$ can be identified with the space of
all linear forms on $D/[D,D]$, i.e. with $Hom(D/[D,D],k)$. A
symmetric linear form on $D$ is $G$-invariant if and only if it
belongs to $Hom(D/[D,D],k)^G$ when regarded as a form on
$D/[D,D]$. Since $D$ is a division algebra, everything follows now
from Proposition \ref{propsymmetriccrossedproducts} and the
considerations before this Proposition for $V=D/[D,D]$.
\end{proof}

\begin{lemma} \label{commutatorextension}
Let $A$ be a $k$-algebra and let $k\subset K$ be an extension of
fields. Then $[K\ot_kA,K\ot_kA]=K\ot_k[A,A]$, where $K\ot_kA$ is
regarded as a $K$-algebra.
\end{lemma}
\begin{proof}
It follows immediately from the relation $[\alpha\ot a,\beta\ot
b]=\alpha\beta \ot [a,b]$ for any $\alpha,\beta\in K$ and any
$a,b\in A$.
\end{proof}

A consequence of the previous Lemma, applied to the extension
$k\subseteq \overline{k}$, where $\overline{k}$ denotes the
algebraic closure of $k$, is the following.

\begin{corollary} \label{dimensioncommutator} (\cite[Lemma 3 and footnote 14]{asano})
Let $D$ be a finite dimensional $k$-division algebra, and let
$\ell$ be the center of $D$. Then ${\rm dim}_{\ell}\; [D,D]={\rm
dim}_{\ell}\; D-1$.
\end{corollary}

\begin{corollary} \label{corollary1}
Let $D$ be a finite dimensional $k$-division algebra with center
$\ell$, such that ${\rm char}\, k$ does not divide ${\rm
dim}_{\ell}D$. Then $1\notin [D,D]$.
\end{corollary}
\begin{proof}
If $1\in [D,D]$, then $1\in \overline{\ell}\ot _{\ell}[D,D]=
[\overline{\ell}\ot_{\ell}D,\overline{\ell}\ot_{\ell}D]$ by Lemma
\ref{commutatorextension}. By \cite[Theorem 15.1]{lam1} we have
$\overline{\ell}\ot _{\ell}D\simeq M_m(\overline{\ell} )$, where
$m^2={\rm dim}_{\ell}D$. Then $1\in
[M_m(\overline{\ell}),M_m(\overline{\ell})]$, which is not
possible, since $tr(1)=m\neq 0$ in $\overline{\ell}$, and the
trace of any commutator in $M_m(\overline{\ell})$ is zero.
\end{proof}

\begin{remark} {\rm (1) If ${\rm char}\, k=0$, the Corollary above applies to any finite
dimensional division algebra. The same happens if ${\rm char}\,
k=p>0$, provided that $k$ is a perfect field. This follows from a
result of Albert which says that the Brauer group of a perfect
field has no $p$-torsion, see the comment in \cite[Theorem
2.5]{fgg}.\\
(2) It is possible that $1\in [D,D]$ for certain finite
dimensional division algebras $D$. One such example is a
quaternion division algebra in characteristic 2. Let
$k=\mathbb{F}_2(t)$ be the field of rational fractions in one
indeterminate over the field with two elements. Then the
4-dimensional $k$-algebra $D$ with basis $\{ 1,u,v,w\}$, subject
to relations
$$u^2+u=1,\;\; v^2=t,\;\; w=uv=v(u+1)$$
is a division algebra, and $1=uv-vu\in [D,D]$. The structure of
this division algebra is discussed in detail in \cite{conrad}.}
\end{remark}

Now we can prove the following, which is a key step towards
Theorem A.

\begin{theorem} \label{theoremgradeddivisionalgebra}
Let $A$ be a finite dimensional $G$-graded division algebra. Then
$A$ is graded symmetric.
\end{theorem}
\begin{proof}
Let $z\in D\setminus [D,D]$. By Corollary
\ref{dimensioncommutator}, $D=\ell z +[D,D]$. Let $g\in G$. Since
$g\cdot [D,D]=[D,D]$, we see that $g\cdot z\notin [D,D]$, so there
exist $\alpha_g\in \ell^*$ and $c_g\in [D,D]$ such that
$z=\alpha_g(g\cdot z)+c_g$. Now if $g,h\in G$, we have \bea
z&=&\alpha_g(g\cdot z)+c_g\\
&=&\alpha_g (g\cdot (\alpha_h(h\cdot z)+c_h))+c_g\\
&=&\alpha_g(g\cdot \alpha_h)(gh\cdot z)+\alpha_g(g\cdot
c_h)+c_g\eea showing that $c_{gh}=\alpha_g(g\cdot c_h)+c_g$ and
$\alpha_{gh}=\alpha_g(g\cdot \alpha_h)$. The latter relation just
says that the family $(\alpha_g)_{g\in G}$ is a 1-cocycle of $G$
in $\ell^*$.\\

Let $\theta:G\ra Gal(\ell/k)$ be the group morphism associated to
the action of $G$ on $\ell$. Let $N=Ker(\theta)$, and pick some
$n\in N$. Then $f:D\ra D$, $f(x)=n\cdot x$, is an automorphism of
the division algebra $D$ with $f(a)=a$ for any $a\in \ell$. By the
Noether-Skolem Theorem, $f$ is inner, so there is $u\in D$ such
that $f(x)=uxu^{-1}$ for any $x\in D$.

Let $\phi:\overline{\ell}\ot_{\ell}D\ra M_m(\overline{\ell})$ be
an isomorphism of $\overline{\ell}$-algebras, where
$m=\sqrt{dim_{\ell}D}$, as in the proof of Corollary
\ref{corollary1}. Then
$$\phi(1\ot (n\cdot x-x))=\phi(1\ot u)\phi(1\ot x)\phi(1\ot
u)^{-1}-\phi(1\ot x)$$ in $M_m(\overline{\ell})$, so then
$\phi(1\ot (n\cdot x-x))$ has trace zero in
$M_m(\overline{\ell})$. This shows that $\phi(1\ot (n\cdot
x-x))\in [M_m(\overline{\ell}),M_m(\overline{\ell})]=\phi
([\overline{\ell}\ot_{\ell}D,\overline{\ell}\ot_{\ell}D])$, and
then $1\ot (n\cdot x-x)\in
[\overline{\ell}\ot_{\ell}D,\overline{\ell}\ot_{\ell}D]=\overline{\ell}\ot
_{\ell}[D,D]$. Thus $n\cdot x-x$ must lie in $[D,D]$ for any $x\in
D$. In particular $n\cdot z-z\in [D,D]$, so $\alpha_n=1$.\\

Now $\alpha_{ng}=\alpha_n(n\cdot \alpha_g)=\alpha_g$ for any $g\in
G$ and $n\in N$, so $\alpha_g$ only depends on the class of $g$ in
the factor group $G/N\simeq \theta(G)$. Then $\theta(G)$ acts on
$\ell^*$ by $\theta(g)\cdot a=g\cdot a$ for any $g\in G$ and $a\in
\ell$, and the family $(\beta_q)_{q\in \theta(G)}$ defined by
$\beta_{\theta(g)}=\alpha_g$ for any $g\in G$, is a 1-cocycle of
$\theta(G)$ in $\ell^*$. Since $\ell/\ell^{\theta(G)}$ is a finite
Galois extension with Galois group $\theta(G)$, we know that
$H^1(\theta(G),\ell^*)=1$ by Hilbert Theorem 90. Thus there exists
$v\in \ell^*$ such that $\beta_q=\frac{q\cdot v}{v}$ for any $q\in
\theta(G)$, and then $\alpha_g=\frac{g\cdot v}{v}$ for any $g\in
G$.\\

Now let $\Phi:\ell \ra D/[D,D]$ be the isomorphism of
$\ell$-vector spaces defined by $\Phi(a)=av\overline{z}$ for any
$a\in \ell$. If $g\in G$ and $a\in \ell$, we have \bea
g\cdot\Phi(a)&=&g\cdot (av\overline{z})\\&=&g\cdot
\overline{avz}\\
&=&\overline{(g\cdot a)(g\cdot v)(g\cdot z)}\\&=&(g\cdot a)(g\cdot
v)\overline{g\cdot z}\\&=&(g\cdot a)(g\cdot
v)\frac{1}{\alpha_g}\overline{z}\\
&=&(g\cdot a)v\overline{z}\\&=&\Phi(g\cdot a)\eea This means that
$\Phi$ is an isomorphism of $kG$-modules, and in order to show
that $(D/[D,D])_G\neq 0$, which by Proposition
\ref{symmetriccoinvariants} will end the proof, it is enough to
show that $\ell_G\neq 0$. This is the same with
$\ell_{\theta(G)}\neq 0$. By the Normal Basis Theorem, $\ell$ is a
free $\ell^{\theta (G)}\theta(G)$-module of rank 1, and then it is
a free $k\theta(G)$-module of rank $[\ell^{\theta(G)}:k]$. Then
$\ell \neq \omega (\theta(G))\ell$, so $\ell_{\theta(G)}\neq 0$.
\end{proof}

\begin{remark}
{\rm (1) Under certain conditions, we can construct a non-zero
coinvariant of $D/[D,D]$ directly from an invariant of
$D/[D,D]=V$. Indeed, if ${\rm char}\, k$ does not divide the order
of $G$, $u\in V$ is non-zero and $g\cdot u=u$ for any $g\in G$,
then $u$ does not lie in $\omega(G)V$. Indeed, if it did, then
$u=\sum_i (g_i\cdot v_i-v_i)$ for some $(g_i)_i$ in $G$ and
$(v_i)_i$ in $V$. Then
$$|G|u=\sum_{g\in G}g\cdot u=\sum_i((\sum_{g\in
G}gg_i)-(\sum_{g\in G}g))\cdot v_i=0$$ Since $|G|\neq 0$ in $k$,
this would imply that $u=0$, a contradiction.  Thus the image of
$u$ in $V_G$ is non-zero.\\
In particular, if $A$ is a finite dimensional $G$-graded division
algebra such that ${\rm char}\, k\nmid {\rm dim}_{\ell}\; A$, then
since ${\rm dim}_{\ell}\; A= |G|{\rm dim}_{\ell}\; D$, we can
apply Corollary \ref{corollary1} and see that 1 (in fact its class
modulo $[D,D]$) is a non-zero coinvariant of $D/[D,D]$.\\
(2) If $1\notin [D,D]$, then the inclusion of $\ell$ in $D$,
followed by the natural projection $D\ra D/[D,D]$, produces
directly an isomorphism of $kG$-modules $\ell\simeq D/[D,D]$. If
$1\in [D,D]$, which by Corollary \ref{corollary1} is a rare
situation, then one needs to construct such an isomorphism as in
the proof of the previous theorem.}
\end{remark}

\section{Proof of Theorem A}

We first consider finite dimensional graded simple algebras $A$,
and we show that they are graded symmetric. We recall that $A$ is
graded simple if 0 and $A$ are the only graded ideals. It is known
that in this case $A$ is a direct sum of minimal graded left
ideals, isomorphic up to a graded shift, see \cite[Section
2.9]{nvo}.

By the graded version of Wedderburn's Theorem, see \cite[Theorem
2.10.10]{nvo}, a finite dimensional graded simple algebra is
isomorphic to a graded algebra of the form
$M_n(\Delta)(\sigma_1,\ldots,\sigma_n)$, for some positive integer
$n$, some finite dimensional graded division algebra $\Delta$, and
some $\sigma_1,\ldots,\sigma_n\in G$. We recall that
$M_n(\Delta)(\sigma_1,\ldots,\sigma_n)$ is just the matrix algebra
$M_n(\Delta)$ with the $G$-grading such that the homogeneous
component of degree $g$ is
$$M_n(\Delta)(\sigma_1,\ldots,\sigma_n)_{g}=\left(
\begin{array}{cccc}
\Delta_{\sigma_1g \sigma_1^{-1}}&\Delta_{\sigma_1g \sigma_2^{-1}}&\ldots&\Delta_{\sigma_1g \sigma_n^{-1}}\\
\Delta_{\sigma_2g \sigma_1^{-1}}&\Delta_{\sigma_2g \sigma_2^{-1}}&\ldots&\Delta_{\sigma_2g \sigma_n^{-1}}\\
\ldots&\ldots&\ldots&\ldots\\
\Delta_{\sigma_ng \sigma_1^{-1}}&\Delta_{\sigma_ng
\sigma_2^{-1}}&\ldots&\Delta_{\sigma_ng \sigma_n^{-1}}
\end{array}
\right)$$

\begin{proposition}  \label{propositionsimple}
A finite dimensional graded simple algebra is graded symmetric.
\end{proposition}
\begin{proof}
Let $A= M_n(\Delta)(\sigma_1,\ldots,\sigma_n)$ be a graded simple
algebra, as above. Since the graded division algebra $\Delta$ is
graded symmetric, there exists a non-zero linear map
$\lambda:\Delta\ra k$  such that $\lambda(ab)=\lambda(ba)$ for any
$a,b\in \Delta$, $\lambda(\Delta_g)=0$ for any $g\neq e$. Define
$\Lambda:A\ra k$ by $\Lambda(x)=\lambda (tr(x))$, where $tr$ is
the usual trace map of $M_n(\Delta)$. Since for $g\neq e$ we have
$\sigma_ig\sigma_i^{-1}\neq e$, we see that $\lambda
(\Delta_{\sigma_ig\sigma_i^{-1}})=0$ for any $i$, showing that
$\Lambda(A_g)=0$.

On the other hand, if $x=(a_{ij})_{1\leq i,j\leq n}$ and
$y=(b_{ij})_{1\leq i,j\leq n}$, then $tr(xy)=\sum
_{i,j}a_{ij}b_{ji}$ and $tr(yx)=\sum_{i,j}b_{ji}a_{ij}$, showing
that $\Lambda(xy)=\lambda(tr(xy))=\lambda(tr(yx))=\Lambda(yx)$.

Finally we show that ${\rm Ker}\; \Lambda$ does not contain
non-zero graded left ideals. Indeed, if $\Lambda(Ax)=0$, where
$x=(a_{ij})_{1\leq i,j\leq n}$ is a homogeneous element in $A$,
then for any $i,j$ we have
$0=\Lambda((be_{ij})x)=\lambda(ba_{ji})$ for any $b\in \Delta$,
where $e_{ij}$ is the usual matrix unit. Thus $\lambda(\Delta
a_{ji})=0$, so $\lambda(\sum_{g\in
G}\Delta_{g^{-1}}(a_{ji})_g)=0$, and then for any $g\in G$ one has
$0=\lambda(\Delta_{g^{-1}}(a_{ji})_g)=\lambda(\Delta
(a_{ji})_g))$. If $(a_{ji})_g\neq 0$, we would obtain $\lambda=0$,
a contradiction. This shows  that $a_{ji}=0$. We conclude that
$x=0$, and then $\Lambda$ makes $A$ a graded symmetric algebra.
\end{proof}

\begin{remark}
{\rm In the case where the classes of $\sigma_1,\ldots,\sigma_n$
in $G/Z(G)$ are equal, in particular in the case where $G$ is
abelian, there is an isomorphism of graded algebras $\Delta\ot
M_n(k)(\sigma_1,\ldots,\sigma_n)\simeq
M_n(\Delta)(\sigma_1,\ldots,\sigma_n)$, given by the usual algebra
isomorphism $\Delta\ot M_n(k)\simeq M_n(\Delta)$. The grading on
$M_n(k)(\sigma_1,\ldots,\sigma_n)$ is such that the homogeneous
component of degree $g$ is the span of all matrix units $e_{ij}$
for which $\sigma_i^{-1}\sigma_j=g$. This is a good grading on the
matrix algebra $M_n(k)$, and it is known that it is graded
symmetric, see \cite[Example 6.4]{dnn}. In this case, the fact
that $ M_n(\Delta)(\sigma_1,\ldots,\sigma_n)$  is graded symmetric
also follows from the fact that the tensor product of two graded
symmetric algebras is graded symmetric.}
\end{remark}

Let now $A$ be a finite dimensional graded semisimple algebra.
Then $A$ is a finite direct product of graded simple algebras, see
\cite[Section 2.9]{nvo}. Using Proposition \ref{propositionsimple}
and the obvious fact that a finite direct product of graded
symmetric algebras is graded symmetric, we obtain that any finite
dimensional graded semisimple algebra is graded symmetric, which
ends the proof of Theorem A.

\section{The center of a graded division algebra}

We note that the center of a symmetric algebra is not necessarily
symmetric, as the following examples show.

\begin{example}
{\rm (1) If $A$ is a finite dimensional algebra, and $M$ is a left
$A$, right $A$-bimodule, the space $A\oplus M$ has an algebra
structure with the multiplication given by
$(a,m)(a',m')=(aa',am'+ma')$; this is called the trivial extension
of $A$ and $M$. If $M=A^*$, the dual space of $A$, with the usual
bimodule structure, the trivial extension $A\oplus A^*$ is a
symmetric algebra, see \cite[Example 16.60]{lam2}. It is easy to
check that $Cen(A\oplus A^*)=Cen(A)\oplus V$, where
$$V=\{ a^*\in A^*|\; a^*([A,A])=0\}\simeq A/[A,A]$$
Thus $Cen(A\oplus A^*)$ is the trivial extension $Cen(A)\oplus
A/[A,A]$.

Assume that $k$ has characteristic different from 2, and let $A$
be the 4-dimensional algebra generated by $c$ and $x$, subject to
relations
$$c^2=1,\; x^2=0,\; xc=-cx.$$
This is just the underlying algebra of Sweedler's Hopf algebra.
Then $Cen(A)=k$ and $[A,A]=<x,cx>$, so $Cen(A\oplus A^*)\simeq
k[u,v]/(u^2,uv,v^2)$, a commutative local algebra having more than
one minimal ideal. We conclude that $Cen(A\oplus A^*)$ is not even
Frobenius, see for example \cite[Exercise 14, page 114]{lam2}.

Finally note that if $D$ is a finite dimensional $k$-division
algebra, then $D/[D,D]$ has dimension 1 over the center $\ell$ of
$D$, so the center of $D\oplus D^*$ is just the trivial extension
$\ell \oplus \ell^*$, which is a symmetric $k$-algebra.

(2) Let $A=kS_3$ be the group algebra of the symmetric group
$S_3$. The center $Cen(A)$ of $A$ has a $k$-basis $\{ 1, x, y\}$,
where $x=\sigma +\sigma^2$ and $y=\tau +\sigma\tau +\sigma^2\tau$,
where $\tau$ is a transposition and $\sigma$ is a cycle of length
3; indeed, it is well known that a basis of a group algebra
consists of the sums of group elements in the conjugacy classes of
the group. The basis elements satisfy the relations
$$x^2=x+2,\; y^2=3x+3,\; xy=yx=2y$$
Using the notation in \cite[page 451]{lam2}, if $\al=
(\al_1,\al_2,\al_3)\in k^3$, then the associated paratrophic
matrix is
$$P_\al= \left( \begin{array}{ccc} \al_1 & \al_2 & \al_3\\
\al_2 & 2\al_1+\al_2 & 2\al_3\\
 \al_3 & 2\al_3 & 3\al_1+3\al_2 \end{array} \right)$$
Then
$det(P_\al)=3(\al_2-2\al_1)(\al_1+\al_2+\al_3)(\al_3-\al_1-\al_2)$.
In the case where $k$ has characteristic 3, we get that $P_\al$ is
singular for any $\al$, so $Cen(A)$ is not Frobenius by
\cite[16.82]{lam2}. If $k$ has characteristic $\neq 3$, then
$Cen(A)$ is Frobenius (or equivalently, symmetric), since for
$\al_1=\al_3=0$ and $\al_2\neq 0$, $P_\al$ is nonsingular.}
\end{example}

By the second example above, we see that the center of a graded
division algebra is not necessarily a symmetric algebra. However,
Theorem B shows that the center of a finite dimensional graded
division algebra is symmetric, provided that the order of
$G$ is not zero in $k$.\\

{\bf Proof of Theorem B:}  We keep the same notation  as in
Section \ref{sectiondivisionalgebras}. Let $R=C_A(A_e)=\{ a\in
A|ab=ba \mbox{ for any }b\in A_e\}$, which is a graded division
subalgebra of $A$. We have that $R_e=C_A(A_e)\cap
A_e=Cen(A_e)=\ell$.

The weak action of $G$ on $A$ induces an action of $G$ on $R$,
more precisely $g\cdot r=u_gru_g^{-1}$ for any $g\in G$ and $r\in
R$; here $(u_g)_{g\in G}$ is a family of invertible elements of
$A$, with $u_g$ of degree $g$.  This action of $G$ on $R$ is just
the Miyashita-Ulbrich action in the particular case of a graded
division algebra. The invariant subalgebra with respect to this
action is $R^G=Cen(A)$.

Let $\mu:\ell \ra k$ be a $k$-linear map such that $\mu(1)\neq 0$,
and let $\lambda:\ell \ra k$, $\lambda(a)=\sum_{g\in G}\mu (g\cdot
a)$. Clearly $\lambda$ is $k$-linear and symmetric, $\lambda
(g\cdot a)=\lambda(a)$ for any $g\in G$, $a\in \ell$, and
$\lambda$ is non-zero, since $\lambda(1)=|G|\mu(1)$.

Then as in the proof of Proposition
\ref{propsymmetriccrossedproducts} we can construct a non-zero
symmetric linear map $\Lambda:R\ra k$ by $\Lambda(r)=\lambda(r_e)$
for any $r\in R$.
 Moreover,  $\Lambda$ is
invariant with respect to the $G$-action. Indeed, if $g\in G$ and
$r\in R$, then \bea \Lambda (g\cdot r)&=&\Lambda
(u_gru_g^{-1})\\&=&\lambda
(u_gr_eu_g^{-1})\\
&=&\lambda (g\cdot r_e)\\&=& \lambda(r_e)\\&=&\Lambda(r).\eea
 Now using
\cite[Exercise 33, page 457]{lam2}, we obtain that $R^G=Cen(A)$ is
symmetric. \\

{\bf Acknowledgment} We thank an anonymous referee, whose comments
and suggestions on a previous version of this paper helped us to
prove Theorem A in the present general version. We also thank
Andrei M\u{a}rcu\c{s} for providing us the references
\cite{harris} and \cite{marcus}.

\end{document}